\documentclass[12pt]{article}
\usepackage{amsmath, amsthm, amsfonts}
\def\fullpage {
\addtolength{\topmargin}{-2 cm}
\addtolength{\oddsidemargin}{-0.9cm} \addtolength{\textwidth}{+2 cm}
\addtolength{\textheight}{+4 cm}}
\fullpage
\newtheorem{thm}{Theorem}
\newtheorem*{thmchernoff}{Chernoff Bound}
\newtheorem*{setup}{Setup}

\newtheorem{claim}[thm]{Claim}

\theoremstyle{remark}
\newtheorem{remark}[thm]{Remark}
\theoremstyle{definition}
\newtheorem*{definition}{Definition}
\newcommand{\Ex}{\mathop{\bf E\/}}
\newcommand{\I}{\mathop{\bf I\/}}

\def \Saarbrucken {{Saarbr\"{u}cken}}

\begin{document}
\title{Counting independent sets in hypergraphs}
\author{Jeff Cooper\thanks{Department of Mathematics, Statistics, and Computer
Science, University of Illinois at Chicago, IL 60607;  email:
jcoope8@uic.edu} , 
Kunal Dutta\thanks{Algorithms and Complexity Department, Max Planck 
Institute for Informatics, 
\Saarbrucken, Germany. (part of this work was done at: 
Indian Statistical Institute, New Delhi, India);  
email: kdutta@mpi-inf.mpg.de}, 
Dhruv Mubayi\thanks{Department of Mathematics, Statistics, and Computer
Science, University of Illinois at Chicago, IL 60607;  
Research partially supported by NSF grants DMS 0969092 and 1300138; email:
mubayi@uic.edu}
}
\maketitle

\begin{abstract}
Let $G$ be a triangle-free graph with $n$ vertices and average degree $t$.
We show that $G$ contains at least
\[
e^{(1-n^{-1/12})\frac{1}{2}\frac{n}{t}\ln t (\frac{1}{2}\ln t-1)}
\]
independent sets. 
This improves a recent result of the first and third authors \cite{countingind}.
In particular, it implies that as $n \to \infty$, every triangle-free graph
on $n$ vertices has at least $e^{(c_1-o(1)) \sqrt{n} \ln n}$
independent sets, where $c_1 = \sqrt{\ln 2}/4 = 0.208138..$.
Further, we show that for all $n$, there exists a 
triangle-free graph with $n$ vertices which
has at most $e^{(c_2+o(1))\sqrt{n}\ln n}$ 
independent sets, where $c_2 = 1+\ln 2 = 1.693147..$.
This disproves a conjecture from \cite{countingind}.

Let $H$ be a $(k+1)$-uniform linear hypergraph with $n$ vertices
and average degree $t$. We also show that there exists a constant $c_k$ such that
the number of independent sets in $H$ is at least
\[
e^{ c_{k} \frac{n}{t^{1/k}}\ln^{1+1/k}{t} }.
\]
This is tight apart from the constant $c_k$ and
generalizes a result of Duke, Lefmann, and R\"odl \cite{uncrowdedrodl},
which guarantees the existence of an independent set
of size $\Omega(\frac{n}{t^{1/k}} \ln^{1/k}t)$.
Both of our lower bounds follow from a more general statement,
which applies to hereditary properties of hypergraphs.

\end{abstract}
\section{Introduction}
An independent set in a graph $G = (V,E)$ is a set $I \subset V$ of
vertices such that no two vertices in $I$ are adjacent. The independence
number of $G$, denoted $\alpha (G)$, is the size of the largest
independent set in $G$. 
\begin{definition}Given a graph $G$, $i(G)$ is the number of independent sets in $G$.
\end{definition}
In \cite{trifreeaks}, Ajtai, Koml\'os, and
Szemer\'edi gave a semi-random algorithm for finding an independent set of size
at least $\frac{n}{100t}\ln t$ in any triangle-free graph $G$ with
$n$ vertices and average degree $t$. 
By analyzing their algorithm, the first and third authors \cite{countingind}
recently showed that for any such graph,
\begin{equation}\label{cmbound}
i(G) \geq 2^{\frac{1}{2400} \frac{n}{t} \log^2_2 t}.
\end{equation}
As a consequence, they proved that every triangle-free graph has at least 
$2^{\Omega(\sqrt{n}\ln n)}$ independent sets and conjectured that this 
could be improved to $2^{\Omega(\sqrt{n}\ln^{3/2}n)}$, based on the best 
constructions of Ramsey graphs by Kim \cite{trifreekim}.

In this paper, we give a simpler proof of \eqref{cmbound}, which substantially improves 
the constant in the exponent and avoids any analysis of the algorithm in
\cite{trifreeaks}. Further, we show that our bound is not far from 
optimal, by disproving the conjecture in \cite{countingind} and constructing
a triangle-free graph with at most $2^{O(\sqrt{n}\ln n)}$ independent sets.
The construction is obtained by modifying the graph obtained by the triangle-free process. Our bounds  follow from the detailed  analysis of this process by Bohman-Keevash \cite{bohmankeevash13}
and Fiz Pontiveros-Griffiths-Morris \cite{trifreelong13}.

All logarithms are to the base $e$, unless explicitly mentioned otherwise.

\begin{thm}\label{thmtrifreegrphs}
Let $G$ be a triangle-free graph with $n$ vertices and average degree $t$.
Then
\[
i(G) \geq \max\{e^{(1-n^{-1/12})\frac{1}{2}\frac{n}{t}\ln t (\frac{1}{2}\ln(t)-1)}, 2^t\}.
\]
Consequently, for every triangle-free graph $H$ on $n$ vertices,
\[
i(H) \geq e^{(1-o(1)) \frac{\sqrt{n\ln 2}\ln n}{4}}.
\]
\end{thm}
The constant in the exponent above is $\sqrt{\ln 2}/4 \approx 0.2081$. As we show below it is not far from optimal as we have an upper bound with exponent $1+\ln 2\approx 1.693$.
\begin{thm}\label{thmtrifreegrphsupper}
For all $n$, there exists a triangle-free graph $G$ on $n$ vertices with
\[
i(G) \leq e^{(1+o(1))(1+\ln 2)\sqrt{n}\ln n}.
\] 
\end{thm}

Using random graphs, one can show that for $t<n^{1/3}$, there is a
triangle-free graph $G$ with independence number at most $(2n/t)\ln t$.
Consequently,
\[
i(G) \leq
\sum_{i=1}^{\alpha(G)} \binom{n}{i} 
\leq 2 \binom{n}{\alpha(G)}
\leq 2 \left(\frac{t e}{2\ln t}\right)^{\frac{2n}{t}\ln t}
< 2 e^{\ln{(te)} \frac{2n}{t}\ln t}
= e^{(1+o(1))\frac{2n}{t}\ln^2 t},
\]
so the constant in the exponent of Theorem 
\ref{thmtrifreegrphs} is within a factor of $8$ of the
best possible constant.


\subsection{Linear hypergraphs}
Fix $k \geq 1$. Using the semi-random method, 
Ajtai, Koml\'os, Pintz, Spencer, and Szemer\'edi \cite{uncrowdedakpss} 
showed that there exists $c_k$ such that every
$(k+1)$-uniform hypergraph $H$ with $n$ vertices, average degree $t$, and girth 5
satisfies $\alpha(H) \geq c_k \frac{n}{t^{1/k}} \ln^{1/k}t$. 
A hypergraph is \emph{linear} (or has girth 3) if any two edges intersect
in at most one vertex.
Duke, Lefmann, and R\"odl \cite{uncrowdedrodl} (using the result of \cite{uncrowdedakpss})
showed that there exists $c'_k$ such that 
every linear $(k+1)$-uniform hypergraph $H$ with $n$ vertices and average degree $t$ 
satisfies
\[
\alpha(H) \geq c'_k \frac{n}{t^{1/k}} \ln^{1/k}t.
\]
This leads to our second theorem.

\begin{thm} \label{thmlowerbdlinhypgr}
Fix $k \geq 1$. There exists $c_k'' > 0$ such that the following holds:
For every $(k+1)$-uniform, linear hypergraph $H$ on $n$ vertices with average degree $t$,
\begin{equation}\label{ilinear}
i(H) \geq e^{ c''_k \frac{n}{t^{1/k}}\ln^{1+1/k}{t} }.
\end{equation}
\end{thm}

In \cite{uncrowdedakpss}, Ajtai, Koml\'os, Pintz, Spencer, and Szemer\'edi
observed that, for infinitely many $t$ and $n$, there exists a 
$(k+1)$-uniform, linear hypergraph $H$ with $n$ vertices, 
average degree $t$, and independence number at most $b'_k \frac{n}{t^{1/k}}\ln^{1/k}t$.
For this hypergraph,
\[
i(H) \leq e^{b''_k \frac{n}{t^{1/k}}\ln^{1+1/k}t},
\]
so \eqref{ilinear} is tight up to the constant in the exponent.

\subsection{Hereditary Properties}

In \cite{twise},
Colbourn, Hoffman, Phelps, R\"odl, and Winkler counted the number of partial 
$S(t,t+1,n)$ Steiner systems by analyzing a semi-random algorithm; 
Using the same techniques, Grable and Phelps \cite{pssgrable} extended their result to partial
$S(t,k,n)$ Steiner systems.
Asratian and Kuzjurin \cite{kuzsteiner} gave a simpler proof of the bound in \cite{pssgrable},
which avoids any algorithm analysis. 
Theorems \ref{thmtrifreegrphs} and \ref{thmlowerbdlinhypgr} both follow
from a more general result (Theorem \ref{thmheredprop} below), which is based on this simpler proof. 
Since our proof avoids any analysis of how 
the independent sets are obtained, we are able to extend the bound in 
\cite{countingind} from triangle-free graphs to
a more general hypergraph setting. Recall that a \emph{hereditary property}
$\mathcal{P}$ of hypergraphs is any set of hypergraphs which is closed under
vertex-deletion.

\begin{thm}\label{thmheredprop}
Fix $k \geq 1$ and $\epsilon \in (0, \frac{4}{k+1})$. 
Let $\mathcal{P}$ be any hereditary hypergraph property.
Suppose there exists a non-decreasing function $f$ so that
every $(k+1)$-uniform hypergraph $H \in \mathcal{P}$ with $n$ vertices and
average degree at most $t$ satisfies
\[
\alpha(H) \geq \frac{n}{t^{1/k}}f(t).
\]
Then there exists $n_0 = n_0(\epsilon)$ such that
every $(k+1)$-uniform hypergraph $H \in \mathcal{P}$ with $n \geq n_0$ vertices
and average degree at most $t < n^k$ satisfies
\[
i(H) \geq e^{\alpha' \frac{n}{t^{1/k}}\ln t},
\]
where
\[
\alpha' = 
\begin{cases}
(1-n^{-\epsilon/21})\frac{1}{k+1}f(t^{\frac{1}{k+1}}), & \text{if $H$ is linear}\\
(1-n^{-\epsilon/21})\frac{1-\epsilon}{k(2k+1)}f(t^{\frac{2k+\epsilon}{2k+1}}), & \text{otherwise}.
\end{cases}
\]
\end{thm}

\begin{remark}
In \cite{ktfreeaeks}, Ajtai, Erd\H{o}s, Koml\'os, and Szemer\'edi asked if every $K_r$-free graph
has independence number at least $\Omega( \frac{n}{t} \ln t)$.
They gave a lower bound of $\Omega(\frac{n}{t} \ln \ln t)$, which 
Shearer \cite{ktfreeshearer} later improved to $\Omega(\frac{n}{t}\frac{\ln t}{\ln \ln t})$ 
for sufficiently large $t$.
Theorem \ref{thmheredprop} implies that if there exists $c_r$ so that every $K_r$-free graph $G$
satisfies $\alpha(G) \geq c_r \frac{n}{t}\ln t$, then 
\[
i(G) \geq
\binom{n}{\Omega(\frac{n}{t}\ln t)}
= e^{\Omega(\frac{n}{t}\ln^2t)}.
\]
\end{remark}

\section{Lower Bounds}\label{secproof}

Theorems \ref{thmtrifreegrphs} and \ref{thmlowerbdlinhypgr}
follow from the linear case of Theorem \ref{thmheredprop}. We will prove Theorem \ref{thmheredprop}
for linear hypergraphs and afterward describe the changes needed for
non-linear hypergraphs.

We first state a version of the Chernoff bound and two
claims, which contain the main differences between the linear
and non-linear cases. The proofs of the claims will follow the
proof of the theorem.

\begin{thmchernoff}[Chernoff bound \cite{molloyreed}]
Suppose $X$ is the sum of $n$ independent variables, each equal to $1$ with 
probability $p$ and 0 otherwise. Then for any $0 \leq t \leq np$,
\[
\Pr(|X-np| > t) < 2e^{-t^2/3np}.
\]
\end{thmchernoff}
\begin{setup}
Fix $k \geq 1$ and $\epsilon \in (0, \frac{4}{k+1})$.
Let $H$ be a $(k+1)$-uniform hypergraph with $n$ vertices,
average degree at most $t < n^k$, 
and maximum degree at most $nt^{\epsilon/8}$. Select each vertex of $H$ independently with 
probability $p$. 
Let $m'$ denote the sum of vertex degrees in the subgraph induced by
the selected vertices.
\end{setup}
\noindent
The next two claims come under the assumption of the setup.
\begin{claim} \label{clmlineardegsum}
If $H$ is linear and $p=t^{\frac{-1}{k+1}}$,
Then for all $n > n_0(\epsilon)$,
\[
\Pr\left[m' > ntp^{k+1} + \frac{ntp^{k+1}}{n^{\epsilon/20}}\right] < n^{-2}.
\]
\end{claim}

\begin{claim} \label{clmgenrldegsum}
If $p=t^{\frac{\epsilon-1}{k(2k+1)}}$, then for all $n > n_0(\epsilon)$,
\[
\Pr\left[m' > ntp^{k+1} + \frac{ntp^{k+1}}{n^{\epsilon/20}}\right] < n^{-2}.
\]
\end{claim}

\begin{proof}[Proof of Theorem \ref{thmheredprop}]
Fix $k \geq 1$ and $\epsilon \in (0, \frac{4}{k+1})$.
Let $H \in \mathcal{P}$ be a $(k+1)$-uniform, linear hypergraph with $n$ vertices and average degree 
at most $t < n^k$. 
We assume $n \geq n_0$, where $n_0$ is chosen implicity so that several inequalities
throughout the proof are satisfied.
We consider two cases. 
In Case 1, we require that the maximum degree of $H$ is at most 
$tn^{\epsilon/8}$, while Case 2 requires the maximum degree
of $H$ to be at least $tn^{\epsilon/8}$.

\noindent
\textbf{Case 1}:
The maximum degree of $H$ is at most $tn^{\epsilon/8}$.

Select each vertex of $H$ independently with probability
$
p = t^{-\frac{1}{k+1}}.
$
Let $H'$ denote the subgraph of $H$ induced by the selected vertices.
Let $n'$ denote the the number of vertices in $H'$.
Since $t < n^k$ and $\epsilon < \frac{4}{k+1}$, 
\[
np = nt^{-\frac{1}{k+1}} > n^{1 - k/(k+1)} = n^{\frac{1}{k+1}} > n^{\epsilon/4}.
\]
By the Chernoff bound,
\begin{equation}\label{change1}
\Pr[|n'-np|> \frac{np}{n^{\epsilon/20}}] \leq 2e^{-np/3n^{\epsilon/20}} < n^{-2}.
\end{equation}
Let $m'$ denote the sum of vertex degrees in $H'$.
By linearity of expectation, 
$$\Ex[m']=ntp^{k+1}.$$
Set $\lambda = n^{-\epsilon/20}$. 
By Claim \ref{clmlineardegsum},
\begin{equation}\label{change2}
\Pr[m' > (1+\lambda)ntp^{k+1}] < n^{-2}.
\end{equation}

Therefore, by the union bound, with probability at least $1-2n^{-2}>1-1/n$,
$H'$ satisfies both 
$$m' \leq (1+\lambda)ntp^{k+1}$$
and 
$$n' \geq (1-\lambda)np.$$
Let $t' = (1+3\lambda)tp^k$.
Then with probability at least $1-1/n$,
$H'$ has average degree at most
\[
m'/n'
\leq \frac{(1+\lambda)ntp^{k+1}}{(1-\lambda)np}
\leq (1+3\lambda)tp^k
= t'.
\]
Since $\mathcal{P}$ is hereditary, $H' \in \mathcal{P}$.
Thus, with probability at least $1-1/n$, $H'$ has an independent set of size at least 
\begin{align*}
\frac{n'}{t'^{1/k}}f(t')
\geq \frac{(1-\lambda)np}{((1+3\lambda)tp^k)^{1/k}} f( (1+3\lambda)tp^k )
&= \frac{(1-\lambda)n}{(1+3\lambda)^{1/k}t^{1/k}} f( (1+3\lambda)tp^k ) \\
&\geq \frac{(1-\lambda)n}{(1+3\lambda)t^{1/k}} f( (1+3\lambda)tp^k ) \\
&> (1-6\lambda)\frac{n}{t^{1/k}} f( (1+3\lambda)tp^k ) \\
&\geq (1-6\lambda)\frac{n}{t^{1/k}} f(tp^k ),
\end{align*}
where we used that $f$ is non-decreasing in the last inequality.

Let $g = (1-6\lambda)\frac{n}{t^{1/k}} f(tp^k )$.
Suppose $I$ is an independent set in $H$ with at least $g$ 
vertices. Then 
$$\Pr[I\subset V(H')] = p^{|I|} \leq p^g.$$

Let $N$ denote the number of independent sets in $H$ with at least 
$g$ vertices, and let the random variable $N'$ denote the number of 
independent sets in $H'$ with at least $g$ vertices. By Markov's 
inequality,
\[
1-1/n < \Pr[N'\geq 1] \leq \Ex[N'] \leq Np^g \\
                                     = N e^{-g\ln p}
\]
Thus
\begin{align}\label{change3}
N
> (1-1/n) e^{-g\ln p}
&= (1-1/n)e^{(1-6\lambda)\frac{1}{k+1}\frac{n}{t^{1/k}}f(t^{\frac{1}{k+1}}) \ln t} \\
&> (1-1/n)e^{(1-n^{-\epsilon/21})\frac{1}{k+1}\frac{n}{t^{1/k}}f(t^{\frac{1}{k+1}}) \ln t}. \notag
\end{align}

\noindent
\textbf{Case 2}: The maximum degree of $H$ is more than $tn^{\epsilon/8}$.

Let 
   $$K= \{u\in V(H): \deg(u)> tn^{\epsilon/8}/2\}.$$
Let $H'$ denote the subgraph of $H$ induced by $V(H)-K$,
and let $n' = |V(H')|$.
Since
\begin{align*}
\frac{1}{n'}\sum_{v \in V(H')} \deg_{H'}(v) - \frac{1}{n}\sum_{v \in V(H')} \deg_{H}(v)
&\leq (\frac{1}{n'}-\frac{1}{n})\sum_{v \in V(H')} \deg_{H}(v)\\
&\leq (\frac{1}{n'}-\frac{1}{n})n' tn^{\epsilon/8}/2 \\
&= (n-n')\frac{tn^{\epsilon/8}}{2n} \\
&\leq \frac{1}{n}\sum_{v \in K}\deg_H(v),
\end{align*}
the average degree of $H'$ is at most 
\[
\frac{1}{n}\sum_{v \in K}\deg_H(v) + \frac{1}{n}\sum_{v \in V(H')}\deg_H(v)
= \frac{1}{n}\sum_{v \in V(H)}\deg_H(v)
\leq t.
\]
Also, because
$$tn \geq \sum_{u\in V(H)} \deg_H(u) \geq \sum_{u\in K} \deg_H(u) > |K|tn^{\epsilon/8}/2,$$
$|K|<2n^{1-\epsilon/8}$, and so $n' > n(1-2n^{-\epsilon/8}) > n/2^{8/\epsilon}$. 
Thus $H'$ has maximum degree at most $tn^{\epsilon/8}/2 < tn'^{\epsilon/8}$.
Further, since $H$ has maximum degree at least $tn^{\epsilon/8}$
and at most $n^k$,
$t < n^{k-\epsilon/8}$. Hence
$t < n^{k-\epsilon/8} < n'^k.$
Thus Case 1 implies that
\[
i(H') \geq
(1-1/n')e^{(1-6\lambda)\frac{1}{k+1}\frac{n'}{t^{1/k}}f(t^{\frac{1}{k+1}}) \ln t}
> (1-2/n)e^{(1-6\lambda)(1-n^{-\epsilon/8})\frac{1}{k+1}\frac{n}{t^{1/k}}f(t^{\frac{1}{k+1}}) \ln t},
\]
where $\lambda = n'^{-\epsilon/20}$.
We conclude that 
\begin{equation}\label{change4}
i(H) \geq i(H') \geq e^{(1-n^{-\epsilon/21})\frac{1}{k+1}\frac{n}{t^{1/k}}f(t^{\frac{1}{k+1}}) \ln t}.
\end{equation}
\end{proof}

The proof of Theorem \ref{thmheredprop} when $H$ is non-linear is similar.
We set $p = t^{\frac{\epsilon-1}{k(2k+1)}}$. 
Since we still have $np > n^{\epsilon/4}$, \eqref{change1} still holds.
We then use Claim \ref{clmgenrldegsum} instead of Claim \ref{clmlineardegsum}
to prove \eqref{change2}. The proof then proceeds in the same way until we get to \eqref{change3},
where, using the different value of $p$, we instead obtain 
\[
N 
> (1-1/n)e^{(1-6\lambda)\frac{1-\epsilon}{k(2k+1)}\frac{n}{t^{1/k}}f(t^\frac{2k+\epsilon}{2k+1}) \ln t}.
\]
Finally, \eqref{change4} becomes
\[
e^{(1-n^{-\epsilon/21})\frac{1-\epsilon}{k(2k+1)}\frac{n}{t^{1/k}}f(t^\frac{2k+\epsilon}{2k+1}) \ln t}.
\]

\noindent
We now prove Theorem \ref{thmtrifreegrphs} and Theorem \ref{thmlowerbdlinhypgr}.
\begin{proof}[Proof of Theorem \ref{thmtrifreegrphs}]
Shearer \cite{trifreeshearer} showed that every triangle-free graph with $n$ vertices
and average degree $t$ has independence number at least $\frac{n}{t}(\ln(t) -1 )$.
Since being triangle-free is hereditary and graphs are $2$-uniform, linear hypergraphs, 
we may apply Theorem \ref{thmheredprop}
(with $f(t) = \ln(t)-1$) to conclude that for $\epsilon = 21/12 \in (0, 2)$,
there exists $n_0$ such that every triangle-free graph $G$ with $n \geq n_0$ vertices
and average degree at most $t$ satisfies
\[
i(G) \geq e^{(1-n^{-\epsilon/21})\frac{1}{2}\frac{n}{t}\ln t (\frac{1}{2}\ln(t)-1)}
> e^{(1-n^{-1/12})\frac{1}{2}\frac{n}{t}\ln t (\frac{1}{2}\ln(t)-1)}.
\]

Suppose $G$ is a triangle-free graph with $n < n_0$ vertices and average degree $t$.
Choose an integer $r$ so that $rn \geq n_0$. Let $G'$ be the disjoint union of $r$
copies of $G$. Then $i(G') = i(G)^r$, so by the previous paragraph,
\begin{align*}
i(G) = i(G')^{1/r} 
&\geq (e^{(1-(rn)^{-1/12})\frac{1}{2}\frac{rn}{t}\ln t (\frac{1}{2}\ln(t)-1)})^{1/r} \\
&\geq e^{(1-n^{-1/12})\frac{1}{2}\frac{n}{t}\ln t (\frac{1}{2}\ln(t)-1)}.
\end{align*}
This completes the proof of the first bound in Theorem 
\ref{thmtrifreegrphs}. For the second part, consider a triangle-free graph $G$ having average
degree $t$. $G$ contains a vertex $u$ with degree at least $t$.
The neighborhood of $u$ is an independent set, which contains $2^t$ independent
sets. Therefore, every triangle-free graph has at least 
\[
\max\{2^t, e^{(1-n^{-1/12})\frac{1}{2}\frac{n}{t}\ln t (\frac{1}{2}\ln(t)-1)}\}
\]
independent sets. 
This is minimized when $t=(\frac{1}{4}+o(1))\sqrt{n/ \ln 2 }\ln n$, so every
triangle-free graph on $n$ vertices has at least 
\[
2^{(1-o(1)).\frac{\sqrt{n}\ln n}{4\sqrt{\ln 2}}}
= e^{(1-o(1)).\frac{\sqrt{n\ln 2}\ln n}{4}}
\]
independent sets.
\end{proof}

\begin{proof}[Proof of Theorem \ref{thmlowerbdlinhypgr}]
Duke, Lefmann, and R\"{o}dl \cite{uncrowdedrodl} showed that every $(k+1)$-uniform linear
hypergraph with $n$ vertices and average degree at most $t$ has independence
number at least $c'_k \frac{n}{t^{1/k}}\ln^{1/k} t$. Since linearity is a
hereditary property, we may apply Theorem \ref{thmheredprop} (with $f(t) = c'_k \ln^{1/k}t$)
to conclude that for $\epsilon = \frac{3}{k+1} \in (0, \frac{4}{k+1})$, there exists $n_0$ such that
every $(k+1)$-uniform linear hypergraph $H$ with $n \geq n_0$ vertices satisfies
\[
i(H) \geq e^{(1-n^{-1/(7(k+1))})\frac{c'_k}{k+1} \frac{1}{(k+1)^{1/k}} \frac{n}{t^{1/k}}\ln^{1+1/k}t}
> e^{c''_k \frac{n}{t^{1/k}}\ln^{1+1/k}t}.
\]

If $H$ is a $(k+1)$-uniform linear hypergraph with $n < n_0$ vertices,
then we proceed in the same way as in the proof of Theorem \ref{thmtrifreegrphs}.
\end{proof}

It only remains to prove the claims stated at the beginning of this section.
We first prove Claim \ref{clmlineardegsum}. We will use the following 
theorem of Kim and Vu \cite{kimvu}:
\begin{thm}\label{kimvu}
Suppose $F$ is a hypergraph such that $W = V(F)$ and $|f| \leq s$ for all $f \in F$. Let
\[
	Z = \sum_{f \in F} \prod_{i \in f}z_i,
\]
where the $z_i$, $i \in W$ are independent random variables taking values in $[0, 1]$.
For $A \subset W$ with $|A| \leq s$, let
\[
	Z_A = \sum_{f \in F: f \supset A} \prod_{i \in f-A}z_i.
\]
Let $M_A = \Ex[Z_A]$ and $M_j = \max_{A: |A| \geq j}M_A$ for $j \geq 0$. Then there exists
positive constants $a = a(s)$ and $b = b(s)$ such that for any $\lambda > 0$, 
\[
	\Pr[|Z-\Ex[Z]| \geq a\lambda^s\sqrt{M_0 M_1}] \leq b|W|^{s-1}e^{-\lambda}.
\]
\end{thm}
\begin{proof}[Proof of Claim \ref{clmlineardegsum}]
Apply Theorem \ref{kimvu} with $F = H$ and $\Pr[z_i = 1] = p = t^{-\frac{1}{k+1}}$.
Note first that 
\[
M_{\emptyset} \leq ntp^{k+1} = nt^{1-1} = n.
\]
Since the maximum degree of $H$ is at most $tn^{\epsilon/8}$,
\[
M_{\{u\}} \leq t n^{\epsilon/8} p^k = n^{\epsilon/8}t^{\frac{1}{k+1}}
\]
for any $u \in V(G)$.
By linearity, for any $A \subset V(G)$ with $|A| \geq 2$,
\[
\Ex[Z_{A}] \leq p^{k+1-|A|} \leq 1.
\]
Since $t \leq n^k$ and $\epsilon < \frac{4}{k+1}$,  
$n \geq n^{\epsilon/8}t^{\frac{1}{k+1}}$.
Further, $n^{\epsilon/8}t^{\frac{1}{k+1}} \geq 1$. Therefore $M_0 \leq n$
and $M_1 \leq n^{\epsilon/8} t^{1/(k+1)}$. Theorem \ref{kimvu} therefore 
implies that there exist constants $a = a(k)$ and $b = b(k)$ such that
\[
\Pr[|m' - \Ex[m']| > a ((k+3) \ln n)^{k+1}\sqrt{ntp^{k+1} tn^{\epsilon/8}p^k}]
\leq b n^{k} e^{-(k+3)\ln n}.
\]
Since $t \leq n^k$ and $\epsilon < \frac{4}{k+1}$,
\[
\sqrt{ntp^{k+1} tn^{\epsilon/8}p^k} 
= \frac{ntp^{k+1}}{n^{1/2 - \epsilon/16}p^{1/2}} 
\leq \frac{ntp^{k+1}}{n^{\epsilon/16}}.
\]
Thus, since $\Ex[m'] \leq ntp^{k+1}$,
\begin{align*}
\Pr[m' > ntp^{k+1} + \frac{ntp^{k+1}}{n^{\epsilon/20}}]
&< \Pr[m' > \Ex[m'] + a((k+3)\ln n)^{k+1} \frac{ntp^{k+1}}{n^{\epsilon/16}}] \\
&\leq b n^{k} e^{-(k+3)\ln n} \\
&< n^{-2}.
\end{align*}
\end{proof}

\noindent
To prove Claim \ref{clmgenrldegsum}, we will
apply the following theorem of Alon, Kim, and Spencer \cite{AKS97}:
\begin{thm}\label{alonkimspencer}
Let $X_1,\dots,X_n$ be independent random variables with
\[
\Pr[X_i = 0] = 1-p_i \text{ and } \Pr[X_i = 1] = p_i.
\]
For $Y = Y(X_1, \dots, X_n)$, suppose that
\[
|Y(X_1, \dots, X_{i-1},1,X_{i+1},\dots,X_n)-Y(X_1, \dots, X_{i-1},0,X_{i+1},\dots,X_n)| \leq c_i
\]
for all $X_1, \dots, X_{i-1}, X_{i+1}, \dots, X_n$, $i = 1, \dots, n$. Then for
\[
\sigma^2 = \sum_{i=1}^n p_i(1-p_i)c_i^2
\]
and a positive constant $\alpha$ with $\alpha \max_i c_i < 2 \sigma^2$,
\[
\Pr[|Y - \Ex[Y]| > \alpha) \leq 2e^{-\frac{\alpha^2}{4\sigma^2}}.
\]
\end{thm}
\begin{proof}[Proof of Claim \ref{clmgenrldegsum}]
Recall that $p = t^{\frac{\epsilon-1}{k(2k+1)}}$.
The random variable $m'$ is determined by the $n$ independent, indicator random
variables $\I[v \in V(H')]$. Each of these affects $m'$ by at most 
$\deg(v) \leq tn^{\epsilon/8}$.
Set $\alpha = \frac{ntp^{k+1}}{n^{\epsilon/16}}$ and $\sigma^2 = n^{1+\epsilon/4}p(1-p) t^2$.
Note that $\alpha tn^{\epsilon/8} \leq 2 \sigma^2$. 
Also, because $t \leq n^k$,
\[
\frac{\alpha^2}{4\sigma^2}
= \frac{np^{2k+1}}{16n^{\epsilon/4+\epsilon/8}(1-p)}
\geq \frac{np^{2k+1}}{16n^{\epsilon/4+\epsilon/8}}
= \frac{nt^{\frac{\epsilon-1}{k}}}{16n^{3\epsilon/8}}
\geq \frac{n^\epsilon}{16n^{3\epsilon/8}}
= n^{5\epsilon/8}/16.
\]
Since $\Ex[m'] \leq ntp^{k+1}$, Theorem \ref{alonkimspencer} implies
\[
\Pr[m' > ntp^{k+1} + \frac{ntp^{k+1}}{n^{\epsilon/20}}]
< \Pr[m' > \Ex[m'] + \frac{ntp^{k+1}}{n^{\epsilon/16}}] \leq 2e^{-n^{5\epsilon/8}/16}
< n^{-2}.
\]
\end{proof}

\section{Upper Bound for Triangle-free Graphs}

In this section we prove Theorem \ref{thmtrifreegrphsupper}.
We use the results of Bohman-Keevash \cite{bohmankeevash13} 
and Fiz Pontiveros-Griffiths-Morris \cite{trifreelong13}
on the triangle-free graph process:
Let $G$ be the maximal graph in which the triangle-free
process terminates.
\begin{thm}[Bohman-Keevash, Fiz Pontiveros-Griffiths-Morris] \label{thmbkmax}
  With high probability, every vertex of $G$ has degree 
  $d\leq (1+o(1)) \sqrt{\frac{1}{2}n\ln n}$, and independence number
  $\alpha \leq (1+o(1)) \sqrt{2n\ln n}$.
\end{thm}

Let $r=\frac{1}{2}\ln n$. Construct the graph $G'$ from $G$
as follows:
\paragraph{Construction of $G'$:}
We take the strong graph product of $G$ and $\bar{K_r}$, the empty
graph on $r$ vertices.
Replace each vertex $v$ of $G$ by a copy $C_v$ of $\bar{K_r}$.
Introduce a complete bipartite graph between all the vertices of 
$C_v$ and $C_u$ if and only if $\{u,v\}\in E(G)$. We obtain the graph
$G'$.
Notice that 
$|V(G')|= N = \frac{1}{2}n\ln n$.

Define the function $f:V(G')\to V(G)$, such that given any 
$i\in C_u\subset V(G')$, $f(i) = u$. For a set $S \subset V(G')$,
define $f(S) = \bigcup_{i\in S} \{f(i)\}$.

\begin{claim}\label{lemmaconstr}
  For every $S \subset V(G')$, $S$ is independent 
  only if $f(S)$ is independent in $G$. Further $|S| \leq r|f(S)|$.
\end{claim}

\begin{proof}
  Given an independent set 
  $I\subset G'$, consider $i,j\in I$. Clearly, if $f(i)\neq f(j)$, then
  $f(i),f(j)$ are not adjacent in $G$, by the construction. 
  Further, if $f(i)=f(j)$, then $i,j$ must belong to some copy of 
  $\bar{K_r}$ in $G'$.
\end{proof}

\begin{proof}[Proof of Theorem \ref{thmtrifreegrphsupper}]
  We shall show that $G'$ is the required graph.
  By Claim \ref{lemmaconstr},
  \begin{align*}
    i(G') &\leq \sum_{I \subset G: I \mbox{ ind. set}} 
    2^{r|I|} \\
    &\leq \alpha \binom{n}{\alpha} 2^{\frac{\alpha\ln n}{2}} \\
    &\leq e^{\ln \alpha + \alpha\ln(ne/\alpha) + \frac{\alpha \ln n (\ln 2)}{2}}.
  \end{align*}
To finish the proof, note that
\begin{align*}
\ln \alpha + \alpha\ln(ne/\alpha) + \frac{\alpha \ln n (\ln 2)}{2}
&= (1+o(1))\frac{\alpha \ln n}{2} + \frac{\alpha \ln n (\ln 2)}{2} \\
&= (\frac{1+\ln 2}{2} + o(1))\alpha \ln n \\
&\leq (\frac{1+\ln 2}{2} + o(1))\sqrt{2n \ln n} \ln n \\
&= (1 + \ln 2 + o(1))\sqrt{N} \ln N.
\end{align*}
\end{proof}
    
\bibliographystyle{amsplain}
\bibliography{bib}
\end{document}